\documentclass[reqno,a4paper,11pt]{amsproc}
\usepackage{amsmath,amscd,amsfonts,amssymb}
\usepackage[T2A]{fontenc}
\usepackage{mathrsfs,dsfont}
\usepackage{fourier}
\usepackage[pdfborder={0 0 0},pdfborderstyle={},colorlinks=true,linkcolor=black,citecolor=black,urlcolor=blue]{hyperref}

\DeclareRobustCommand{\cyrtext}{\fontencoding{T2A}\selectfont\def\encodingdefault{T2A}}
\DeclareRobustCommand{\textcyr}[1]{\leavevmode{\cyrtext #1}}

\def\R{\mathbb{R}}

\def\Z{\mathbb{Z}}
\def\T{\mathbb{T}}
\def\1{\mathds{1}}
\def\eps{\varepsilon}

\newcommand\zalpha{\Z \alpha + \Z}
\newcommand\lam[1]{\Lambda(\alpha,#1)}
\newcommand\E{E}
\newcommand\D{\mathcal{D}}

\renewcommand\leq{\leqslant}
\renewcommand\geq{\geqslant}

\newcommand\ft[1]{\widehat #1}
\newcommand\dotprod[2]{\langle #1 , #2 \rangle}
\newcommand\mes{\operatorname{mes}}
\newcommand\bmo{\operatorname{BMO}}

\theoremstyle{plain}
\newtheorem{theorem-main}{Theorem}
\newtheorem{corollary-main}{Corollary}
\newtheorem{theorem}{Theorem}[section]

\newtheorem{lemma}[theorem]{Lemma}
\newtheorem{corollary}[theorem]{Corollary}
\newtheorem*{theorem-a}{Theorem A}
\newtheorem*{theorem-b}{Theorem B}

\theoremstyle{definition}
\newtheorem*{definition*}{Definition}
\newtheorem*{remark*}{Remark}

\newenvironment{enumerate-math}
{\begin{enumerate}
\addtolength{\itemsep}{5pt}
\renewcommand\theenumi{(\roman{enumi})}}
{\end{enumerate}}

\begin{document}

\title[Exponential Riesz bases, discrepancy and BMO]
{Exponential Riesz bases, discrepancy of irrational\\
rotations and BMO}

\author{Gady Kozma}
\address{Department of Mathematics, Weizmann Institute of Science, Rehovot 76100, Israel.}
\email{gady.kozma@weizmann.ac.il}

\author{Nir Lev}
\address{Department of Mathematics, Weizmann Institute of Science, Rehovot 76100, Israel.}
\email{nir.lev@weizmann.ac.il}

\subjclass[2000]{42C15, 11K38}

\keywords{Riesz bases, Quasicrystals, Discrepancy, Bounded mean oscillation}

\begin{abstract}
We study the basis property of systems of exponentials with frequencies
belonging to `simple quasicrystals'. We show that a diophantine condition
is necessary and sufficient for such a system to be a Riesz basis in $L^2$
on a finite union of intervals. For the proof we extend to $\bmo$ a theorem
of Kesten about the discrepancy of irrational rotations of the circle.
\end{abstract}

\maketitle


\section{Introduction}
\label{section:introduction}

\subsection{Sampling and interpolation}
A band-limited signal is an entire function $F$ of exponential type,
square-integrable on the real axis. According to the classical Paley-Wiener
theorem, $F$ is the Fourier transform of an $L^2$-function supported by a
bounded (measurable) set $S \subset \R$, which is called the spectrum of
$F$. We shall denote by $PW_S$ the Paley-Wiener space of all functions
$F \in L^2(\R)$ which are Fourier transforms of functions from $L^2(S)$,
\[
F(t) = \int_{S} f(x) \, e^{-2 \pi i t x} \, dx, \quad f \in L^2(S).
\]

A discrete set $\Lambda \subset \R$ is called a \emph{set of sampling}
for $PW_S$ if every signal with spectrum in $S$ can be reconstructed in
a stable way from its `samples' $\{F(\lambda), \lambda \in \Lambda\}$,
that is, there are positive constants $A,B$ such that the inequalities
\begin{equation}
\label{eq:sampling}
A \|F\|_{L^2(\R)} \leq
\Big( \sum_{\lambda \in \Lambda} |F(\lambda)|^2 \Big)^{1/2} \leq
B \|F\|_{L^2(\R)}
\end{equation}
hold for every $F \in PW_S$. Also, $\Lambda$ is called a \emph{set of
interpolation} for $PW_S$ if every data $\{c_{\lambda}\} \in \ell^2(\Lambda)$
can be ``transmitted'' as samples, which means that there exists at least one
function $F \in PW_S$ such that $F(\lambda) = c_\lambda$ $(\lambda \in \Lambda)$.

The sampling and interpolation properties can also be formulated in terms
of the exponential system
\[
\E(\Lambda) = \{\exp 2 \pi i \lambda t, \; \lambda \in \Lambda\}.
\]
The sampling property of $\Lambda$ means that $\E(\Lambda)$ is a \emph{frame}
in the space $L^2(S)$, while the interpolation holds when $\E(\Lambda)$ is
a \emph{Riesz-Fischer} system in this space. It follows that $\Lambda$ is
a set of both sampling and interpolation if and only if $\E(\Lambda)$ forms
a \emph{Riesz basis} in $L^2(S)$. See \cite{young} for a detailed discussion.

Sampling and interpolation may also be discussed in the periodic setting.
If $S$ is a measurable subset of the circle group $\T = \R / \Z$, then one can
consider the frame or Riesz-Fischer properties in $L^2(S)$ of the exponential
system $\E(\Lambda)$, where $\Lambda \subset \Z$.

\subsection{Density}
The right inequality in \eqref{eq:sampling} follows from the separation condition
\[
\inf_{\lambda, \lambda' \in \Lambda, \; \lambda \neq \lambda'} |\lambda - \lambda'| > 0
\]
which is also necessary for the interpolation, so this condition will always be
assumed below. The lower and upper uniform densities of a separated set
$\Lambda$ are defined respectively by
\[
\D^-(\Lambda) = \lim_{r \to \infty} 
\min_{a \in \R} \frac{\# (\Lambda \cap (a, a+r))}{r} \, ,
\quad
\D^+(\Lambda) = \lim_{r \to \infty} 
\max_{a \in \R} \frac{\# (\Lambda \cap (a, a+r))}{r}.
\]
Landau obtained in \cite{landau:necessary} necessary conditions for sampling and
interpolation in terms of these densities:

\emph{If $\Lambda$ is a set of sampling for $PW_S$, then $\D^{-}(\Lambda) \geq \mes S$.}

\emph{If $\Lambda$ is a set of interpolation for $PW_S$, then $\D^{+}(\Lambda) \leq \mes S$.}

\noindent
Here $S$ is a bounded measurable set, and $\mes S$ is the Lebesgue measure of $S$.

For ``regularly'' distributed sequences the two densities above coincide,
and their common value, denoted $\D(\Lambda)$, is called the
\emph{uniform density} of $\Lambda$. It follows that:

\emph{If $\E(\Lambda)$ is a Riesz basis in $L^2(S)$ then $\Lambda$ has
a uniform density $\D(\Lambda) = \mes S$.}

\subsection{Universality problem}
Olevskii and Ulanovskii posed in 
\cite{olevskii-ulanovskii:universal-1, olevskii-ulanovskii:universal-2}
the following question: is it possible to find a ``universal'' set $\Lambda$
of given density, which provides a stable reconstruction of any signal whose
spectrum has a sufficiently small Lebesgue measure? Similarly, does there exist
$\Lambda$ of given density which is a set of interpolation in every $PW_S$
with $\mes S$ sufficiently large?

It was proved in \cite{olevskii-ulanovskii:universal-1, olevskii-ulanovskii:universal-2}
that no universal set of sampling or interpolation exists if the spectrum $S$
is allowed to be an arbitrary bounded measurable set. On the other hand it
was also proved that under some topological restrictions on the spectra,
universal sampling and interpolation does exist:

\emph{Given any $d > 0$ there is a (separated) set $\Lambda \subset \R$ of
density $\D(\Lambda) = d$, such that:
\begin{enumerate-math}
\item
$\Lambda$ is a set of sampling for $PW_S$ for every compact set
$S \subset \R$, $\mes S < d$.
\item
$\Lambda$ is a set of interpolation for $PW_S$ for every open set
$S \subset \R$, $\mes S > d$.
\end{enumerate-math}}

In fact, this is a consequence of the following result:

\begin{theorem-a}[Olevskii and Ulanovskii
\cite{olevskii-ulanovskii:universal-1, olevskii-ulanovskii:universal-2}]
Given any $d > 0$ one can find a discrete set $\Lambda \subset \R$
(a perturbation of an arithmetical progression) such that $\E(\Lambda)$
is a Riesz basis in $L^2(S)$ for every set $S \subset \R$, $\mes S = d$,
which is the union of finitely many disjoint intervals such that the
lengths of these intervals and the gaps between them are commensurable.
\end{theorem-a}

The latter result shows that one can construct a universal exponential
Riesz basis $\E(\Lambda)$, $\Lambda \subset \R$, in the space $L^2(S)$
for a ``dense'' family of sets $S \subset \R$.

\subsection{Simple quasicrystals}
A different construction of universal sampling and interpolation sets,
termed `simple quasicrystals', was presented by Matei and Meyer in the
papers \cite{matei-meyer-quasicrystals, matei-meyer-variant, matei-meyer-simple}.
Their construction is based on the so-called `cut and project' scheme
introduced by Meyer in 1972 (see \cite{meyer-pisot, meyer-algebraic}).
Here we shall restrict ourselves to the periodic setting, in which the
simple quasicrystals take the following form.

Let $\alpha$ be an irrational real number, and consider the sequence of
points $\{n \alpha\}$, $n \in \Z$, on the circle $\T = \R / \Z$. Given
an interval $I = [a,b) \subset \T$ define the following subset of $\Z$,
\[
\lam{I} := \{ n \in \Z : a \leq n \alpha < b \}.
\]
It is well-known that the points $\{n \alpha\}$ are equidistributed
on the circle $\T$. Moreover, they are well-distributed (see
\cite{kuipers-niederreiter}). This implies that the set $\lam{I}$
has a uniform density $\D(\lam{I}) = |I|$, where $|I|$ denotes the
length of the interval $I$.

The set $\lam{I}$ is a set of universal sampling and interpolation:

\begin{theorem-b}[Matei and Meyer \cite{matei-meyer-variant}]
\quad
\begin{enumerate-math}
\item
$\E(\lam{I})$ is a frame in $L^2(S)$ for every compact set $S \subset \T$, $\mes S < |I|$.
\item
$\E(\lam{I})$ is Riesz-Fischer in $L^2(S)$ for every open set $S \subset \T$, $\mes S > |I|$.
\end{enumerate-math}
\end{theorem-b}

A similar result in the non-periodic setting was formulated
and proved in \cite{matei-meyer-quasicrystals, matei-meyer-simple}.
In the paper \cite{matei-meyer-variant} the authors raised the
question of what can be said in the ``limiting case'' when the
measure of $S$ is equal to the density of $\Lambda$.

\subsection{Riesz bases and quasicrystals}
In the present paper we study the following problem. Is it true
that the exponential system $\E(\lam{I})$ is a Riesz basis in 
$L^2(S)$ for certain subsets $S$ of the circle $\T$? Moreover,
does it provide a universal Riesz basis for some ``dense'' family
of sets?

Our first result shows that the question admits a positive answer
provided that a diophantine condition, relating $\alpha$ and the
length of the interval $I$, holds:

\begin{theorem-main}
\label{thm:main-yes}
Let $|I| \in \zalpha$. Then the exponential system $\E(\lam{I})$ is a
Riesz basis in $L^2(S)$ for every set $S \subset \T$, $\mes S = |I|$,
which is the union of finitely many disjoint intervals whose lengths
belong to $\zalpha$.
\end{theorem-main}

This result will be proved in a more general form in Section
\ref{section:riesz-bases}. It follows that if $|I| \in \zalpha$
then the exponential system $\E(\lam{I})$ is a universal Riesz basis
for a ``dense'' family of sets, and, in particular, that properties
(i) and (ii) in Theorem B can be streng\-thened as follows:

\begin{corollary-main}
\label{cor:main-yes}
Let $|I| \in \zalpha$. Suppose that $U \subset \T$ is an open
set, $K$ is compact, $K \subset U$ and $\mes K < |I| < \mes U$.
Then one can find a set $S$, $K \subset S \subset U$, which is
a finite union of intervals and such that $\E(\lam{I})$ is a
Riesz basis in $L^2(S)$.
\end{corollary-main}

Theorem \ref{thm:main-yes} also provides a new result about existence of
exponential Riesz bases on finite unions of intervals. This problem was
studied in \cite{bezuglaya-katsnelson, lev, lyubarskii-seip,
lyubarskii-spitkovsky, seip} and in other papers, and it is still open in
general. We discuss this in more details in Section \ref{section:remarks}.

Our second result complements the picture by clarifying the role of the 
diophantine assumption $|I| \in \zalpha$ in Theorem \ref{thm:main-yes}.
It turns out that this condition is not only sufficient, but also necessary,
for the simultaneous sampling and interpolation property on ``multiband
spectra''.

\begin{theorem-main}
\label{thm:main-no}
Suppose that
$|I| \notin \zalpha$. Then $\E(\lam{I})$ is not a Riesz basis in $L^2(S)$,
for any set $S \subset \T$ which is the union of finitely many intervals.
\end{theorem-main}

\subsection{Discrepancy and Kesten's theorem}
The proofs of the results above are based on the connection of the
problem to the theory of equidistribution and discrepancy for the
irrational rotation of the circle.

Let $\alpha$ be a fixed irrational number. Given an interval $I \subset
\T$, we denote by $\nu(n,I)$ the number of integers $0 \leq k \leq n-1$
such that $k \alpha \in I$. The equidistribution of the points
$\{n \alpha\}$ on the circle $\T$ means that
\[
\lim_{n \to \infty} \frac{\nu(n, I)}{n} = |I|,
\]
for every interval $I$. A quantitative measurement of this equidistribution
is given by the \emph{discrepancy function}, defined by
\[
D(n,I) = \nu(n,I) - n |I|.
\]
Thus we have $D(n,I) = o(n)$ as $n \to \infty$, and, in fact, it is
not difficult to show that this estimate holds uniformly with respect
to the interval $I$.

Better estimates for the discrepancy can be obtained based on diophantine
properties of the number $\alpha$. For example, if $\alpha$ is a quadratic
irrational then $D(n, I) = O(\log n)$, and this estimate is uniform with
respect to $I$. On the other hand, it is known that for every $\alpha$
the lower bound
\[
\sup_{I \subset \T} |D(n,I)| > c \log n
\]
holds for infinitely many $n$'s, where $c$ is a positive absolute 
constant (see \cite{kuipers-niederreiter}). Compare also to the case
where $\alpha$ and $I$ are random \cite{kesten:i, kesten:ii}.

It was discovered, however, that for certain special intervals $I$ the
discrepancy is bounded. Hecke proved in \cite{hecke} that if $|I| \in
\zalpha$ then $D(n,I) = O(1)$ as $n \to \infty$. It was conjectured
by Erd\"{o}s and Sz\"{u}sz \cite{erdos} that also the converse to 
Hecke's result should be true. This conjecture was confirmed by Kesten
\cite{kesten} who proved that if $D(n,I) = O(1)$ then $|I| \in \zalpha$.
For further developments see \cite{furstenberg-keynes-shapiro, oren, petersen}.

We will prove an extension of Kesten's theorem for \emph{bounded mean oscillation}
of the discrepancy. We say that a sequence of complex numbers $\{c_n\}$ has
bounded mean oscillation, $\{c_n\} \in \bmo$, if for every $n < m$ one has
\begin{equation}
\label{eq:def-bmo-1}
\frac1{m-n} \sum_{k=n}^{m-1} \Big| c_k - \frac{c_{n} + \dots + c_{m-1}}{m-n} \Big|
\leq M,
\end{equation}
for some constant $M$ independent of $n$ and $m$. Certainly, every bounded
sequence belongs to $\bmo$. On the other hand, it is well-known that $\bmo$
contains also unbounded sequences, for example $c_n = \log n$ is such a sequence.

We will prove the following generalization of Kesten's theorem.

\begin{theorem-main}
\label{thm:main-kesten}
Let $\alpha$ be an irrational number, and $I \subset \T$ be an interval.
If the sequence $\{D(n,I)\}$, $n=1,2,3,\dots$, belongs to $\bmo$, then
$|I| \in \zalpha$.
\end{theorem-main}

This result along with Hecke's theorem reveals a dichotomy: either
the discrepancy $D(n, I)$ is bounded, or it does not even have
bounded mean oscillation. In Section \ref{section:ergodic} we will
actually prove an extension of Theorem \ref{thm:main-kesten} for
more general sets than single intervals (such as unions of several
intervals), but, on the other hand, we cannot show that such a
dichotomy still holds in this case.

\subsection{Meyer's duality}
The link between the sampling and interpolation problems for the sets $\lam{I}$
and the theory of discrepancy for irrational rotations is an idea due to Meyer,
which we refer to as the `duality principle'. Meyer's duality principle
(see Section \ref{section:duality}) enables us to reduce the problem about 
exponential Riesz bases in $L^2(S)$ to a similar problem in $L^2(I)$, where
$I$ is a single interval. It is then possible to invoke known results about
exponential Riesz bases in $L^2(I)$.

The proof of Theorem \ref{thm:main-yes} (Section \ref{section:riesz-bases})
consists of three main ingredients: Meyer's duality principle, a theorem of Avdonin
\cite{avdonin} about exponential Riesz bases in $L^2(I)$ (an extension of Kadec's
$1/4$ theorem) and Hecke's result about discrepancy. In order to prove Theorem
\ref{thm:main-no} (Section \ref{section:necessity}) we combine the duality principle
with a theorem due to Pavlov \cite{pavlov}, which describes completely the exponential
Riesz bases in $L^2(I)$. Pavlov's theorem allows us to conclude that the discrepancy
must be in $\bmo$, and we can then apply an extension of Theorem \ref{thm:main-kesten}
(Section \ref{section:ergodic}).

\subsection{Acknowledgements}
We are grateful to A.\ M.\ Olevskii for introducing us to the concept of
quasicrystals and for his suggestions which improved the presentation of
this paper. We thank Itai Benjamini for referring us to Kesten's paper.
We also thank Evgeny Abakumov, Jordi Marzo, Joaquim Ortega-Cerd\`{a},
Ron Peled, Omri Sarig, Mikhail Sodin and Armen Vagharshakyan for helpful
comments and discussions.


\section{Meyer's duality principle}
\label{section:duality}

\subsection{}
Let $\alpha$ be a fixed irrational number. For any set $S \subset \T$ one can,
in principle, consider the ``Meyer set'' based on $\alpha$ and $S$, defined by
\begin{equation}
\label{eq:def-lam}
\lam{S} = \{n \in \Z : n \alpha \in S\}.
\end{equation}
Meyer discovered a duality phenomenon connecting the sampling and interpolation
properties of the sets $\lam{S}$. This duality allowed to prove the results
about sampling and interpolation such as Theorem B mentioned in the introduction.
Meyer's duality principle is the starting point for our approach as well.

It will be convenient to introduce the following terminology.

\begin{definition*}
A set $S \subset \T$ will be called a \emph{multiband set} if it is the union
of finitely many disjoint intervals. We will say that a multiband set $S$ is 
\emph{left semi-closed} if each one of the intervals contains its left 
endpoint but not its right endpoint.
Equivalently, $S$ is left semi-closed if the indicator function $\1_S$ is 
continuous from the right. We also define a \emph{right semi-closed} multiband
set, in a similar way. Finally, $S$ will be called \emph{semi-closed} if it is
either left semi-closed or right semi-closed.
\end{definition*}

We shall need a version of the duality principle which is suitable for our
setting. Let $-\Lambda$ denote the set $\{-n : n \in \Lambda\}$. We will
prove the following:

\begin{lemma}
\label{lemma:duality}
Let $U$ be a semi-closed multiband set, and $V$ be a (not necessarily semi-closed)
multiband set. Then:
\begin{enumerate-math}
\item
If $\E(\lam{V})$ is a frame in the space $L^2(U)$, then $\E(-\lam{U})$ is a 
Riesz-Fischer system in $L^2(V)$.
\item
If $\E(\lam{V})$ is a Riesz-Fischer system in $L^2(U)$, then $\E(-\lam{U})$
is a frame in $L^2(V)$.
\end{enumerate-math}
\end{lemma}

An immediate consequence of Lemma \ref{lemma:duality} is:

\begin{corollary}
\label{cor:duality}
Let $U$ and $V$ be two multiband sets, where $U$ is semi-closed. 
If the exponential system $\E(\lam{V})$ is a Riesz basis in $L^2(U)$,
then $\E(- \lam{U})$ is a Riesz basis in $L^2(V)$.
\end{corollary}

\subsection{}
We choose and fix a function $\varphi(x)$ on $\R$, infinitely smooth,
with compact support contained in the interval $(0, 1)$, and such that 
$\int_{\R} |\varphi(x)|^2 \, dx = 1$. Let
\[
\ft{\varphi}(\xi) = \int \varphi(x) \, e^{-2 \pi i \xi x} \, dx,
\quad \xi \in \R,
\]
be the Fourier transform of $\varphi$, which is a smooth and rapidly decreasing
function. For each $0 < \eps < 1$ define a function $\varphi_\eps$ on the circle $\T$ by
\[
\varphi_\eps(t) = \frac1{\sqrt{\eps}} \; \varphi (t/\eps), \quad 0 \leq t < 1.
\]
It follows that $\varphi_\eps$ is an infinitely smooth function on $\T$, supported
by $(0, \eps)$, such that $\int_{\T} |\varphi_\eps(t)|^2 \, dt = 1$, and
the Fourier coefficients of $\varphi_\eps$ are given by
\[
\ft{\varphi}_\eps(n) = \sqrt{\eps} \; \ft{\varphi}(\eps n), \quad n \in \Z.
\]

The following two lemmas are essentially due to Matei and Meyer
\cite{matei-meyer-quasicrystals, matei-meyer-simple}.

\begin{lemma}
\label{lemma:uniform-distribution}
For every Riemann integrable function $f$ on $\T$,
\[
\lim_{\eps \to 0} \; \sum_{n \in \Z} \big| f(n \alpha) \, 
\ft{\varphi}_\eps(n) \big|^2 = \int_{\T} |f(t)|^2 \, dt.
\]
\end{lemma}

\begin{lemma}
\label{lemma:summation}
Let $\{c_n\}$ be a sequence of complex numbers in $\ell^1(\Z)$. Then
\[
\lim_{\eps \to 0} \; \int_{S} \Big| \sum_{n \in \Z}
c_n \, \varphi_\eps (t - n \alpha) \Big|^2 \, dt
= \sum_{n \in \lam{S}} |c_n|^2,
\]
for every left semi-closed multiband set $S$.
\end{lemma}

Here, as always, $\alpha$ is an arbitrary irrational number. Lemma
\ref{lemma:uniform-distribution} is a consequence of the equidistribution
of the points $\{n \alpha\}$ on $\T$. Lemma \ref{lemma:summation} is due
to the fact that $\varphi_\eps(t - n \alpha)$ is supported by a small right
neighborhood of the point $n \alpha$. For the proof see \cite{matei-meyer-simple}.
Certainly, one can also get a version of Lemma \ref{lemma:summation} for right
semi-closed multi\-band sets by choosing the function $\varphi$ supported by the
interval $(-1,0)$ instead of $(0,1)$.

We will also need the following well-known fact:

\begin{lemma}[See \cite{young}, p.\ 155]
\label{lemma:interpolation}
The exponential system $\E(\Lambda)$ is Riesz-Fischer in $L^2(S)$ if and only if
there is a positive constant $C$ such that the inequality
\[
\sum_{\lambda \in \Lambda} |c_\lambda|^2 \leq
C \int_{S} \Big| \sum_{\lambda \in \Lambda} c_\lambda \, e^{2 \pi i \lambda t} \Big|^2 \, dt
\]
holds for every finite sequence of scalars $\{c_\lambda\}$.
\end{lemma}

\subsection{}
Here we give the proof of Lemma \ref{lemma:duality}. It is based on the
proof from \cite{matei-meyer-quasicrystals, matei-meyer-simple} but we
find it useful to provide the reader with a self contained proof. Below
we will assume that $U$ is a left semi-closed multiband set, but as we
have remarked one can easily adapt the proof to the case when $U$ is
right semi-closed.

\begin{proof}[Proof of Part (i) of Lemma \ref{lemma:duality}]
Suppose that $\E(\lam{V})$ is a frame in the space $L^2(U)$.
It is to be proved that $\E(-\lam{U})$ is a Riesz-Fischer system in $L^2(V)$.
By Lemma \ref{lemma:interpolation} it will be enough to show that, 
if $f$ is a trigonometric polynomial
\begin{equation}
\label{eq:polynomial}
f(t) = \sum_{n} c_n \, e^{- 2 \pi i n t}
\end{equation}
such that $c_n = 0$ unless $n \alpha \in U$, then
\begin{equation}
\label{eq:want-part-i}
\sum_{n} |c_n|^2 \leq C \int_{V} |f(t)|^2 \, dt.
\end{equation}
Given such a trigonometric polynomial $f$ we define
\[
F_\eps(t) = \sum_{n \in \Z} c_n \, \varphi_\eps (t - n \alpha),
\]
where $\varphi$ is any function as in the previous section. Since
$U$ is left semi-closed, it follows that $F_\eps$ is supported 
by $U$ if $\eps$ is sufficiently small. Since $\E(\lam{V})$ is
a frame in $L^2(U)$, there is a constant $C$ such that
\begin{equation}
\label{eq:sampling-interval}
\int_{U} |F_\eps(t)|^2 \, dt \leq C \sum_{n \in \lam{V}} |\ft{F}_\eps(n)|^2.
\end{equation}

Again we take the limit as $\eps \to 0$. Since $U$ is left semi-closed, Lemma
\ref{lemma:summation} implies that the left hand side of \eqref{eq:sampling-interval}
converges to $\sum |c_n|^2$. On the other
hand, it is easy to see that $\ft{F}_\eps(n) = f(n \alpha) \, \ft{\varphi}_\eps(n)$.
The right hand side of \eqref{eq:sampling-interval} can therefore be written as
\[
\sum_{n \in \Z} \big| f(n \alpha) \, \1_{V}(n \alpha) \, \ft{\varphi}_\eps(n) \big|^2,
\]
and by Lemma \ref{lemma:uniform-distribution} it converges to $\int_{V}
|f(t)|^2 \, dt$ as $\eps \to 0$. This proves the first part of Lemma 
\ref{lemma:duality}.
\end{proof}

\begin{proof}[Proof of Part (ii) of Lemma \ref{lemma:duality}]
Suppose that $\E(\lam{V})$ is a Riesz-Fischer system in $L^2(U)$.
It is to be proved that $\E(-\lam{U})$ is a frame in $L^2(V)$,
that is, we must show that the inequality
\begin{equation}
\label{eq:want-part-ii}
\int_{V} |f(t)|^2 \, dt \leq C \sum_{n \in \lam{U}} |\ft{f}(-n)|^2
\end{equation}
holds for every $f \in L^2(V)$. Since $V$ is a multiband set, it is actually
enough to verify \eqref{eq:want-part-ii} for every infinitely smooth
function $f$ supported by $V$. Given such $f$, define
\[
F_\eps(t) := \sum_{n \in \Z} f(n \alpha) \, 
\ft{\varphi}_\eps(n) \exp 2 \pi i n t, \quad t \in \T.
\]
The fact that $f$ is supported by $V$ implies that only exponentials
from $\E(\lam{V})$ have their coefficient non-zero in the series above.
Since $\E(\lam{V})$ is a Riesz-Fischer system in $L^2(U)$, it follows
from Lemma \ref{lemma:interpolation} that there is a constant $C$ such that
\begin{equation}
\label{eq:interpolation-interval}
\sum_{n \in \Z} \big| f(n \alpha) \, \ft{\varphi}_\eps(n) \big|^2
\leq C \int_{U} |F_\eps(t)|^2 \, dt.
\end{equation}

Now we take the limit of \eqref{eq:interpolation-interval} as $\eps \to 0$.
Lemma \ref{lemma:uniform-distribution} implies that the left hand side
of \eqref{eq:interpolation-interval} tends to $\int |f(t)|^2 \, dt$.
On the other hand, substituting $f$ with its Fourier expansion in the
definition of $F_\eps$, it is easy to see that
\[
F_\eps(t) = \sum_{n \in \Z} \ft{f}(-n) \, \varphi_\eps (t - n \alpha).
\]
The coefficients $\{\ft{f}(-n)\}$ belong to $\ell^1$, since $f$ is smooth.
Since $U$ is left semi-closed we may use Lemma \ref{lemma:summation},
which implies that the limit as $\eps \to 0$ of the right hand side of
\eqref{eq:interpolation-interval} is equal to the right hand side of
\eqref{eq:want-part-ii}. This proves the second part of Lemma \ref{lemma:duality}.
\end{proof}


\section{Exponential Riesz bases on multiband sets}
\label{section:riesz-bases}

\subsection{}
In this section we prove Theorem \ref{thm:main-yes}. Let $I$ be an interval
on $\T$ which is either left or right semi-closed. We will show that if the
(necessary) diophantine condition $|I| \in \zalpha$ is satisfied, then the
exponential system $\E(\lam{I})$ is a universal Riesz basis in $L^2(S)$ for
a family of multiband sets $S$. In fact we will prove a somewhat more general
result than formulated in the introduction.

\begin{theorem}
\label{thm:bases-exist}
Suppose that $|I| \in \zalpha$. Then $\E(\lam{I})$ is a Riesz basis in 
$L^2(S)$ for every set $S \subset \T$, $\mes S = |I|$, which satisfies
the following condition: the indicator function $\1_S$ can be expressed
as a finite linear combination of indicator functions of intervals $I_1,
\dots, I_N$ whose lengths belong to $\zalpha$, that is,
\begin{equation}
\label{eq:linear-combination}
\1_S(t) = \sum_{j=1}^{N} c_j \, \1_{I_j}(t), \quad c_j \in \Z,
\quad |I_j| \in \zalpha \quad (1 \leq j \leq N).
\end{equation}
\end{theorem}

Condition \eqref{eq:linear-combination} is certainly satisfied
if $S$ is the union of finitely many \emph{disjoint} intervals
with lengths in $\zalpha$. However, notice that other configurations
are also possible. For example, consider the set $S$ of the form
\[
S = I_1 \setminus I_2, \quad I_2 \subset I_1, \quad |I_1|,|I_2| \in \zalpha,
\]
which certainly satisfies the condition \eqref{eq:linear-combination},
but which is the union of two disjoint intervals whose lengths do not
necessarily belong to $\zalpha$.

\subsection{}
We may suppose, with no loss of generality, that $S$ is a left
semi-closed multiband set, or, equivalently, that the intervals
$I_j$ in the representation \eqref{eq:linear-combination} are
left semi-closed. The condition \eqref{eq:linear-combination}
plays its key role in the following lemma.

\begin{lemma}
\label{lemma:g-exists}
Let $S$ be a left semi-closed multiband set satisfying \eqref{eq:linear-combination}.
Then there is a bounded function $g: \T \to \R$, continuous from the right and
with finitely many jump discontinuities, such that
\begin{equation}
\label{eq:g-exists}
\1_{S}(t) - \mes S = g(t) - g(t + \alpha), \quad t \in \T.
\end{equation}
\end{lemma}

\begin{proof}
According to \eqref{eq:linear-combination} the indicator function $\1_S$
is a linear combination of the functions $\1_{I_j}$ $(1 \leq j \leq N)$.
As equation \eqref{eq:g-exists} is linear as well, it will be enough
to prove the lemma in the case when $S$ is each one of the intervals
$I_j$. In other words, we may suppose that $S$ is a single interval whose
length belongs to $\zalpha$. Moreover, by rotation, we may suppose that
$S$ is an interval whose left endpoint is zero.

Let therefore $S = [0, n \alpha)$, where $n \in \Z$. For simplicity 
we shall suppose that $n$ is a positive integer, as the case when $n$ is negative 
is very similar. Let us denote by $\theta(t)$ the $1$-periodic function on $\R$,
defined by $\theta(t) = t$ for $0 \leq t < 1$. Considered as a function on the
circle $\T$, this is a piecewise linear function, with slope $+1$, and with a jump
discontinuity of magnitude $-1$ at $t=0$. Set
\[
g(t) := \sum_{k=1}^{n} \theta(t - k \alpha), \quad t \in \T,
\]
then $g$ is a bounded function, continuous from the right, and with finitely
many jump discontinuities. We have
\[
g(t) - g(t + \alpha) = \theta(t - n \alpha) - \theta(t).
\]
Observe that the function on the right hand side has the following properties:
it has a jump of magnitude $+1$ at $t = 0$ and another jump of magnitude $-1$ at
$t = n \alpha$, it has derivative zero at all other points, and has zero integral
on $\T$. These properties determine the function uniquely as $\1_S(t) - \mes S$,
and so \eqref{eq:g-exists} is established.
\end{proof}

\begin{remark*}
If $S$ is a multiband set, then the condition \eqref{eq:linear-combination}
is not only sufficient, but also necessary, for the existence of a bounded measurable
function $g(t)$ satisfying \eqref{eq:g-exists}. This is a consequence of a result
due to Oren \cite{oren} (see also \cite{schoissengeier}).
\end{remark*}

\subsection{}
We turn to the proof of Theorem \ref{thm:bases-exist}. The first step in the
proof, based on Meyer's duality principle, is to reduce the problem about
exponential Riesz bases in $L^2(S)$ to a similar problem in $L^2(I)$, where
$I$ is a single interval. This allows us then to use a theorem of Avdonin
\cite{avdonin} on exponential Riesz bases in $L^2(I)$. Below we formulate
a special case of Avdonin's theorem, in a form which will be convenient
in our setting.

\begin{theorem}[Avdonin \cite{avdonin}]
\label{thm:avdonin}
Let $I \subset \R$ be a bounded interval, and let
\begin{equation}
\label{eq:avdonin}
\lambda_j = \frac{j + \delta_j + c}{|I|}, \quad j \in \Z,
\end{equation}
where $c$ is a constant, and $\{\delta_j\}$ is a bounded sequence
of real numbers. Suppose that there is a positive integer $N$ such that
\begin{equation}
\label{eq:avdonin-condition}
\sup_{n \in \Z} \; \Big| \frac1{N} \sum_{j=1}^{N} \delta_{n+j} \Big| < \frac1{4}.
\end{equation}
If the sequence $\Lambda = \{\lambda_j, \; j \in \Z\}$ is separated
then $\E(\Lambda)$ is a Riesz basis in $L^2(I)$.
\end{theorem}

Here the separation condition means that 
$\inf_{n \neq m} |\lambda_n - \lambda_m| > 0$.

Note that the famous Kadec $1/4$ theorem (see \cite{young})
corresponds to the special case when $N=1$ in Avdonin's theorem.

\begin{proof}[Proof of Theorem \ref{thm:bases-exist}]
Let $S$ be a multiband set satisfying the condition 
\eqref{eq:linear-combination} and such that $\mes S = |I|$. 
We must prove that $\E(\lam{I})$ is a Riesz basis in $L^2(S)$.
Recall that $I$ is semi-closed by assumption. Hence by
Corollary \ref{cor:duality}, with $U = -I$ and $V = S$, it will
be enough to show that $\E(\lam{S})$ is a Riesz basis for $L^2(-I)$.

There is no loss of generality in assuming that $S$ is left semi-closed. 
Let us enumerate the set $\lam{S}$ in an increasing order,
\[
\lam{S} = \{\lambda_j, \; j \in \Z\}, \quad
\cdots < \lambda_{-1} < 0 \leq \lambda_0 < \lambda_1 < \lambda_2 < \cdots,
\]
and take $g$ to be the function from Lemma \ref{lemma:g-exists}.
By the definition \eqref{eq:def-lam} of $\lam{S}$ and according to 
\eqref{eq:g-exists}, for any two integers $m < n$ we have
\[
\# (\lam{S} \cap [m, n)) = \sum_{k=m}^{n-1} \1_S(k \alpha)
= (n-m) \mes S + g(m \alpha) - g(n \alpha).
\]
Using this with $m=0$, $n = \lambda_j$ $(j \geq 0)$ or with
$m = \lambda_j$, $n=0$ $(j < 0)$ we get
\[
j = \lambda_j \, \mes S + g(0) - g(\lambda_j \alpha), \quad j \in \Z.
\]
Since $\mes S = |I|$, this implies \eqref{eq:avdonin}
with $\delta_j = g(\lambda_j \alpha)$ and $c = - g(0)$.

Observe that the perturbations $\{\delta_j\}$ are bounded, since $g$
is bounded. We may assume, by adding a constant to $g$ if necessary,
that $\int_{S} g(t) \, dt = 0$. Hence
\[
\lim_{m \to \infty} \; 
\frac1{m} \sum_{k=n+1}^{n+m} g(k \alpha) \1_{S}(k \alpha) = 0,
\]
uniformly with respect to $n \in \Z$
(see \cite{meyer-algebraic}, Chapter V, \textsection 6.3).
It is then easy to see that \eqref{eq:avdonin-condition} holds for a
sufficiently large $N$. Moreover, the sequence $\lam{S}$ is clearly
separated, as the elements $\lambda_j$ are distinct integers. So the
proof is concluded by Theorem \ref{thm:avdonin}.
\end{proof}

\begin{proof}[Proof of Corollary \ref{cor:main-yes}]
Let $|I| \in \zalpha$. First we claim that, being given any $\eps > 0$,
one can find real numbers $\gamma_1, \dots, \gamma_s$ in the segment
$(0, \eps)$, such that $\gamma_j \in \zalpha$ for each $j$,
and $\sum \gamma_j = |I|$. Indeed, since $\zalpha$ is a dense
subset of $\R$, it would be possible to choose numbers
$\gamma_1, \dots, \gamma_{s-1} \in (0, \eps) \cap (\zalpha)$
such that the difference
\begin{equation}
\label{eq:difference-gammas}
|I| - \sum_{j=1}^{s-1} \gamma_j
\end{equation}
lies in $(0, \eps)$, and then define $\gamma_s$ to be the value
of \eqref{eq:difference-gammas}, which is also in $\zalpha$.

Let now $U \subset \T$ be an open set, $K$ be compact, $K \subset U$
and $\mes K < |I| < \mes U$. If $\eps$ is sufficiently small then $K$
may be covered by disjoint intervals $I_1, \dots, I_s$ contained in
$U$, such that $|I_j| = \gamma_j$. Taking $S := I_1 \cup \cdots \cup I_s$
we have $K \subset S \subset U$, $\mes S = |I|$, and $\E(\lam{I})$
is a Riesz basis in $L^2(S)$ by Theorem \ref{thm:main-yes}.
\end{proof}


\section{Ergodic sums, BMO and Kesten's theorem}
\label{section:ergodic}

In this section we study the bounded mean oscillation of ergodic sums,
first in an abstract setting, and then for the irrational rotation of
the circle. In particular we will prove an extension of Theorem
\ref{thm:main-kesten}. The result obtained will then allow us to prove
Theorem \ref{thm:main-no} in Section \ref{section:necessity}.

\subsection{}
It will be convenient to start with an abstract setting. Let $H$
be a Hilbert space, and $U$ be a unitary operator on $H$. Given
a vector $f \in H$ we consider its `ergodic sums' defined by
\[
S_n := f + Uf + \cdots + U^{n-1}f.
\]

The von Neumann ergodic theorem asserts that the ratios $S_n / n$ converge
to the projection of $f$ onto the closed subspace of $U$-invariant vectors.
In particular, $f$ is perpendicular to this subspace if and only if 
$\|S_n\| = o(n)$ as $n \to \infty$.

A vector $f$ is called a \emph{coboundary} if there exists $g \in H$
such that $f = g - Ug$. In this case the ergodic sums have the form
$S_n = g - U^n g$. This, of course, implies the boundedness of the ergodic
sums, $\|S_n\| = O(1)$. However, it was proved by Robinson \cite{robinson}
that the latter condition is not only necessary but also sufficient for $f$
to be a coboundary. In fact, as the proof in \cite{robinson} shows, the
condition $(1/N) \sum_{n=1}^{N} \|S_n\|^2 = O(1)$ implies
that $f$ is a coboundary.

We will prove the following 

\begin{theorem}
\label{eq:thm-bounded-variance}
For a vector $f$ to be a coboundary it is necessary and sufficient 
that the numbers $V_N$ (so-called `variances') defined by
\begin{equation}
\label{eq:variance}
V_N := \frac1{N} \sum_{n=1}^{N} \Big\|S_n - \frac{S_1 + \dots + S_N}{N} \Big\|^2
\end{equation}
are bounded for $N=1,2,3,\dots$\,.
\end{theorem}

This result strengthens Robinson's theorem, as the assumption
$V_N = O(1)$ is weaker than the condition $(1/N) \sum_{n=1}^{N}
\|S_n\|^2 = O(1)$. This can be seen from the identity
\[
V_N = \frac1{N} \sum_{n=1}^{N} \big\| S_n \big\|^2 -
\Big\| \frac1{N} \sum_{n=1}^{N} S_n \Big\|^2.
\]

A key tool in \cite{robinson} as well as in our proof of Theorem
\ref{eq:thm-bounded-variance} is the \emph{spectral measure}. Since
the sequence $\{\dotprod{U^n f}{f}\}$ is positive definite, by Herglotz's
theorem (see \cite{katznelson}) there exists a positive, finite measure
$\mu_f$ on the circle $\T$, such that
\begin{equation}
\label{eq:spectral}
\int_{\T} e^{2\pi i n t} \; d\mu_f(t) = \dotprod{U^n f}{f}, \quad n \in \Z.
\end{equation}
The measure $\mu_f$ is called the spectral measure of $f$ with respect to $U$. 
In the paper \cite{robinson} the coboundaries were characterized also in terms
of the spectral measure, in the following way: $f$ is a coboundary if and only
if the integral
\[
\int_{\T} \, \frac{d\mu_f(t)}{\sin^2 \pi t}
\]
(with the integrand taking the value $+ \infty$ at $t=0$) is finite.
Theorem \ref{eq:thm-bounded-variance} is obtained by a combination
of this result and the following

\begin{lemma}
\label{lemma:variance}
For any $f \in H$ we have
\begin{equation}
\label{eq:limit-variance}
\lim_{N \to \infty} V_N = \frac1{4} \int_{\T} \frac{d\mu_f(t)}{\sin^2 \pi t} \; ,
\end{equation}
where the integral on the right hand side may be finite or infinite.
\end{lemma}

\begin{proof}
It follows from \eqref{eq:spectral} that
\[
\|P(U)f\|^2 = \int_{\T} |P(e^{2 \pi i t})|^2 \, d\mu_f(t)
\]
for every polynomial $P(z) = c_0 + c_1 z + \cdots + c_n z^n$. In particular,
using \eqref{eq:variance} we can express the variance $V_N$ in terms of the
spectral measure in the form
\begin{equation}
\label{eq:kernel}
V_N = \int_{\T} Q_N(t) \; d\mu_f(t),
\end{equation}
where $Q_N(t)$ is the trigonometric polynomial defined by
\[
Q_N(t) := \frac1{N} \sum_{n=1}^{N} \Big| \sum_{k=0}^{n-1} e^{2 \pi i k t}
- \frac1{N} \sum_{m=1}^{N} \sum_{k=0}^{m-1} e^{2 \pi i k t} \Big|^2.
\]
By evaluation of the inner sums we get
\[
Q_N(t) = \frac1{N} \sum_{n=1}^{N} \Big| \frac{1 - e^{2 \pi i n t}}{1 - e^{2 \pi i t}} 
- \frac1{N} \sum_{m=1}^{N} \frac{1 - e^{2 \pi i m t}}{1 - e^{2 \pi i t}} \Big|^2,
\]
and it follows that
\[
\begin{aligned}
&4 \, Q_N(t) \sin^2 \pi t 
= \frac1{N} \sum_{n=1}^{N} \Big| (1 - e^{2 \pi i n t})
- \frac1{N} \sum_{m=1}^{N} (1 - e^{2 \pi i m t}) \Big|^2\\
&\quad = \frac1{N} \sum_{n=1}^{N} \Big| e^{2 \pi i n t}
- \frac1{N} \sum_{m=1}^{N} e^{2 \pi i m t} \Big|^2
= \frac1{N} \sum_{n=1}^{N} \Big| e^{2 \pi i n t} \Big|^2
- \Big| \frac1{N} \sum_{n=1}^{N} e^{2 \pi i n t} \Big|^2
= 1 - \frac{\sin^2 \pi N t}{N^2 \sin^2 \pi t} \, .
\end{aligned}
\]
We have thus obtained the formula
\[
Q_N(t) = \frac1{4 \sin^2 \pi t}
\left( 1 - \frac{\sin^2 \pi N t}{N^2 \sin^2 \pi t} \right).
\]
The conclusion of \eqref{eq:limit-variance} from this formula is immediate:
in the case when the integral on the right hand side of \eqref{eq:limit-variance}
is finite one should merely apply to \eqref{eq:kernel} the dominated convergence
theorem, while in the case of divergence of the integral in \eqref{eq:limit-variance}
the result follows from Fatou's lemma.
\end{proof}

\subsection{}
\label{subsection:bmo}
We shall now consider the special case when $H = L^2(\T)$ and the
unitary operator $U$ is the irrational rotation,
\[
(Uf)(x) = f(x + \alpha), \quad f \in L^2(\T),
\]
where $\alpha$ is a fixed irrational number. In this case the ergodic
sums have the form
\begin{equation}
\label{eq:ergodic-sums}
S_n(x) = \sum_{k=0}^{n-1} f(x + k \alpha).
\end{equation}

Much attention has been paid to the case when these ergodic sums are
\emph{bounded}, that is, $\sup_{n} |S_n(x)| < \infty$ for some
(or every) $x \in \T$. For more details we refer the reader to the
papers \cite{furstenberg-keynes-shapiro, oren, petersen, schoissengeier}
and to the references therein.

Here we consider the \emph{bounded mean oscillation} of the ergodic
sums. In the next theorem we shall see that for a ``good'' function
$f$, the $\bmo$ behavior of the ergodic sums implies that $f$ is a
coboundary with respect to the rotation by $\alpha$.

\begin{theorem}
\label{thm:bmo-riemann}
Let $\alpha$ be an irrational number, and $f$ be a Riemann integrable
function on $\T$. Suppose that for some fixed $x_0 \in \T$, the sequence 
$\{S_n(x_0)\}$, $n=1,2,3,\dots$, belongs to $\bmo$. Then there exists a
function $g \in L^2(\T)$ such that $f(x) = g(x) - g(x + \alpha)$
almost everywhere.
\end{theorem}

Recall the definition \eqref{eq:def-bmo-1} of the space $\bmo$ given in Section
\ref{section:introduction}. It is well-known that replacing the condition
\eqref{eq:def-bmo-1} with an appropriate $\ell^p$ version yields an equivalent
definition of $\bmo$. This is a consequence of a classical theorem of John and
Nirenberg \cite{john-nirenberg} (see also \cite{garnett}, Chapter VI
\textsection 2) in a version adopted for sequences. In the next lemma we
formulate this fact for $p=2$.

\begin{lemma}[\cite{john-nirenberg}]
\label{lem:john-nirenberg}
$\{c_n\} \in \bmo$ if and only if there is a constant $M$ such that
\begin{equation}
\label{eq:def-bmo-2}
\left(
\frac1{m-n} \sum_{k=n}^{m-1} \Big| c_k - \frac{c_{n} + \dots + c_{m-1}}{m-n} \Big|^2
\right)^{1/2} \leq M
\end{equation}
for every $n < m$.
\end{lemma}

\begin{proof}[Proof of Theorem \ref{thm:bmo-riemann}]
Define
\[
v_N (x) := \left( \frac1{N} \sum_{n=1}^{N} 
\Big| S_n(x) - \frac{S_1(x) + \cdots + S_N(x)}{N} \Big|^2
\right)^{1/2} \quad (x \in \T).
\]
Using the obvious property $S_n(x + j \alpha) = S_{j+n}(x) - S_j(x)$
it follows that
\[
v_N(x + j \alpha) = \left(
\frac1{N} \sum_{n=1}^{N} \Big| S_{j+n}(x) - 
\frac{S_{j+1}(x) + \cdots + S_{j+N}(x)}{N} \Big|^2
\right)^{1/2}.
\]
The assumption that $\{S_n(x_0)\} \in \bmo$ combined with Lemma
\ref{lem:john-nirenberg} thus implies the existence of a
constant $M$ such that
\[
v_N(x_0 + j \alpha) \leq M, \quad j=0,1,2,\dots \,.
\]
The function $v_N$ is therefore bounded by $M$ on a dense subset
of $\T$. But $v_N$ is a Riemann integrable function, since so
is $f$, hence it follows that
\[
\int_{\T} v_N (x)^2 \, dx \leq M^2, \quad N=1,2,\dots \,.
\]
On the other hand, we have $\int v_N (x)^2 \, dx = V_N$,
where $V_N$ is the variance defined by \eqref{eq:variance}.
The variances are therefore bounded, so the proof is concluded
by Theorem \ref{eq:thm-bounded-variance}.
\end{proof}

\subsection{}
We can now prove Theorem \ref{thm:main-kesten}. In fact, this theorem
is a consequence of the following more general result.

\begin{theorem}
\label{thm:bmo-kesten}
Let $\alpha$ be an irrational number, and $S \subset \T$ be a measurable
set whose boundary has Lebesgue measure zero. Let $\nu(n,S)$ denote the
number of integers $0 \leq k \leq n-1$ such that $k \alpha \in S$.
If the sequence $\{\nu(n,S) - n \mes S\}$, $n=1,2,3,\dots$, belongs
to $\bmo$, then $\mes S \in \zalpha$.
\end{theorem}

The proof is a combination of the previous result with an argument due
to Furstenberg, Keynes and Shapiro \cite{furstenberg-keynes-shapiro}
and Petersen \cite{petersen}.

\begin{proof}[Proof of Theorem \ref{thm:bmo-kesten}]
The fact that the boundary of $S$ has Lebesgue measure zero
means that the function $f(x) = \1_S(x) - \mes S$ is Riemann
integrable. Let $S_n(x)$ be the ergodic sums of $f$ defined
by \eqref{eq:ergodic-sums}. Then we have $\{S_n(0)\} \in \bmo$
by assumption, and thus Theorem \ref{thm:bmo-riemann} implies
the existence of a function $g \in L^2(\T)$ such that
$f(x) = g(x) - g(x + \alpha)$ (a.e.). Certainly we may also
suppose that $g$ is real-valued.

Now consider the function $\tau(x) = \exp 2\pi i g(x)$. We have
\[
\tau(x + \alpha) = e^{2\pi i g(x + \alpha)} = e^{2\pi i (g(x)-f(x))} 
= \tau(x) \, e^{2\pi i \mes S} \quad \text{(a.e.)},
\]
since the function $\1_S$ takes integer values. This means 
that $\tau$ is an eigen\-function of the irrational rotation by $\alpha$,
with eigen\-value $\exp 2\pi i \mes S$. However all the eigenvalues are
known to be of the form $\exp 2\pi i j \alpha$, $j \in \Z$, and therefore
$\mes S \in \zalpha$.
\end{proof}


\section{Necessary condition for Riesz bases}
\label{section:necessity}

In this section we prove Theorem \ref{thm:main-no}. We will show that unless
the diophantine condition $|I| \in \zalpha$ holds, the exponential system
$\E(\lam{I})$ cannot be a Riesz basis in $L^2(S)$ for any multiband set $S$.
The proof consists of three main ingredients: Meyer's duality principle,
the results from Section \ref{section:ergodic} about discrepancy and $\bmo$,
and Pavlov's theorem describing the exponential Riesz bases in $L^2(I)$
where $I$ is an interval.

\subsection{}
We start with the formulation of Pavlov's theorem.

Let $f(x)$ be a locally integrable function on $\R$. If $J$
is a bounded interval on $\R$ then
\[
\frac1{|J|} \int_{J} \big| f(x) - f_J \big| \, dx
\]
is called the \emph{mean oscillation} of $f$ over $J$, where
$f_J = |J|^{-1} \int_{J} f(x) \, dx$ denotes the average of
$f$ over $J$. If the mean oscillation is bounded uniformly
over all intervals $J$, then we say that $f$ has bounded
mean oscillation, $f \in \bmo(\R)$.

For a discrete set $\Lambda \subset \R$ we denote by $n_\Lambda(x)$
the `counting function' satisfying
\[
n_\Lambda(b) - n_\Lambda(a) = \#(\Lambda \cap [a,b)), \quad a<b,
\]
which is defined uniquely up to an additive constant.
We will use the following version of Pavlov's theorem
(see \cite[p. 240]{hruscev-nikolskii-pavlov}) formulated
in terms of the counting function $n_\Lambda$.

\begin{theorem}[Hru\u{s}\u{c}ev, Nikol'skii, Pavlov \cite{hruscev-nikolskii-pavlov}]
\label{thm:pavlov}
Let $\Lambda = \{\lambda_n, \, n \in \Z\}$. The exponential system
$\E(\Lambda)$ is a Riesz basis in $L^2(0,a)$, $a > 0$, if and only if
\begin{enumerate-math}
\item
\label{cond:pavlov-separated}
$\Lambda$ is separated, that is, $\inf_{n \neq m} |\lambda_n - \lambda_m| > 0$;
\item
\label{cond:pavlov-bmo}
$f(x) = n_\Lambda(x) - ax$ is a function in $\bmo(\R)$;
\item
\label{cond:pavlov-cancellations}
There is $y > 0$ such that the harmonic continuation $U_f(x+iy)$ of $f$ into the
upper half plane admits the following representation:
\[
U_f(x + iy) = c + \widetilde{u}(x) + v(x), \quad x \in \R,
\]
where $c$ is a constant, $u,v$ are bounded measurable functions, $\|v\|_{\infty} 
< \tfrac1{4}$, and $\widetilde{u}$ is the Hilbert transform of $u$.
\end{enumerate-math}
\end{theorem}

In the proof below we shall exploit only part \ref{cond:pavlov-bmo}
of Theorem \ref{thm:pavlov}. We therefore do not discuss part 
\ref{cond:pavlov-cancellations} in more detail. See 
\cite{hruscev-nikolskii-pavlov} for a complete exposition.

\subsection{}
We can finally prove the necessity of the condition $|I| \in \zalpha$.

\begin{proof}[Proof of Theorem \ref{thm:main-no}]
Suppose that $\E(\lam{I})$ is a Riesz basis in $L^2(S)$, for some
multiband set $S \subset \T$. We may suppose that $S$ is semi-closed.
Corollary \ref{cor:duality}, with $U = S$ and $V = I$, implies
that $\E(\lam{S})$ is a Riesz basis in $L^2(-I)$. Using part
\ref{cond:pavlov-bmo} of Theorem \ref{thm:pavlov} it follows that
the function
\[
f(x) = n_{\lam{S}}(x) - |I| x
\]
belongs to $\bmo(\R)$.

Since $\lam{S}$ is a subset of $\Z$, the function $f$ is linear on each
interval $(n-1, n]$ and has slope $-|I|$ there. It is therefore possible
to write $f$ as the sum of two functions, $f = g + h$, where $g(x)$ is
a piecewise constant function which is equal to $f(n)$ on each interval
$(n-1, n]$, and $h(x)$ is a $1$-periodic function which is linear with
slope $-|I|$ on each such an interval. Thus,
\begin{equation}
\label{eq:fn-bmo}
\frac1{m-n} \sum_{k=n+1}^{m}
\Big| f(k) - \frac{f(n+1) + \cdots + f(m)}{m-n} \Big|
\end{equation}
is the mean oscillation of $g$ over the interval $(n, m]$. But since
$g$ differs from $f$ by a bounded function, also $g \in \bmo(\R)$. This
implies that \eqref{eq:fn-bmo} is bounded uniformly with respect to $n$
and $m$, that is, the sequence $\{f(n)\}$, $n=1,2,3,\dots$, belongs to
$\bmo$.

Recall that Landau's inequalities imply that $\mes S = |I|$.
By adding an appropriate constant to the counting function
$n_{\lam{S}}$ we may assume that $n_{\lam{S}}(0) = 0$. This
means that
\[
f(n) = \nu(n, S) - n \mes S, \quad n=1,2,3,\dots,
\]
where $\nu(n,S)$ denotes the number of integers $0 \leq k \leq n-1$ 
such that $k \alpha \in S$. Since we know that $\{f(n)\} \in \bmo$,
it follows from Theorem \ref{thm:bmo-kesten} that $\mes S \in \zalpha$.
We conclude that $|I| \in \zalpha$ and so the theorem is proved.
\end{proof}


\section{Remarks}
\label{section:remarks}

\subsection{}
If $S$ is a single interval, then the complete description of the
exponential Riesz bases $\E(\Lambda)$ in $L^2(S)$ is given by Pavlov's
theorem. Much less is known, however, in the case when $S$ is the
union of two intervals or more. In fact, it is unknown in general
whether a Riesz basis $\E(\Lambda)$ (where $\Lambda \subset \R$)
in $L^2(S)$ exists at all. This existence has been established in
the following special cases:
\begin{enumerate-math}
\item
$S$ is a finite union of disjoint intervals with commensurable
lengths \cite{bezuglaya-katsnelson, lyubarskii-seip}.
\item
$S$ is the union of two general intervals \cite{seip} (in this
paper there is also a partial result for the union of more than
two intervals).
\end{enumerate-math}

Theorem \ref{thm:bases-exist} exhibits a new family of examples
of multiband sets $S$ with a Riesz basis of exponentials, namely,
every set $S$ with the structure \eqref{eq:linear-combination}.
Moreover, in these examples the Riesz basis constructed consists
of exponentials with integer frequencies.

A result of similar type in the non-periodic setting is given
in \cite{lev}.

We also mention the paper \cite{lyubarskii-spitkovsky} where the
authors construct, for any finite union of intervals, a Riesz
basis of exponentials with complex frequencies lying in a horizontal
strip along the real axis (see also \cite{lyubarskii-seip}).

\subsection{}
We mention several questions which are left open.

\begin{enumerate}
\renewcommand\theenumi{(\alph{enumi})}
\item
\label{item:open-riesz}
\emph{For which multiband sets $S \subset \T$, $\mes S = |I|$,
is the exponential system $\E(\lam{I})$ a Riesz basis in $L^2(S)$?}
\end{enumerate}

In the case when $|I| \in \zalpha$ it was proved that condition
\eqref{eq:linear-combination} is sufficient for $S$ (Theorem
\ref{thm:bases-exist}), but we have not classified completely
all such multiband sets $S$. It was also proved that no such
sets exist if $|I| \notin \zalpha$ (Theorem \ref{thm:main-no}).

\begin{enumerate}
\renewcommand\theenumi{(\alph{enumi})}
\setcounter{enumi}{1}
\item
\label{item:open-bmo}
\emph{For which multiband sets $S \subset \T$ does the
discrepancy $D(n,S) = \nu(n,S) - n \mes S$ (with respect
to the irrational number $\alpha$) have bounded mean
oscillation?}
\end{enumerate}

Theorem \ref{thm:bmo-kesten} says that $\mes S \in \zalpha$
is a necessary condition for that. 

As for the \emph{boundedness} of the discrepancy, a necessary
and sufficient condition is known to be the condition
\eqref{eq:linear-combination}. This result was proved by Oren
in \cite{oren} (see also \cite{schoissengeier}). The following
question therefore seems natural:

\begin{enumerate}
\renewcommand\theenumi{(\alph{enumi})}
\setcounter{enumi}{2}
\item
\label{item:open-oren}
\emph{If $S$ is a multiband set such that $\{D(n,S)\} \in
\bmo$, must it have the structure \eqref{eq:linear-combination}?}
\end{enumerate}

If the answer to \ref{item:open-oren} is affirmative, then this
would extend the dichotomy given by Theorem \ref{thm:main-kesten}
for single intervals: either the discrepancy is bounded, or it
does not even have bounded mean oscillation. It would also settle
questions \ref{item:open-riesz} and \ref{item:open-bmo}.

\subsection{}
In connection with `universal' sampling and interpolation, the
result of Olevskii and Ulanovskii mentioned above (Theorem A of
Section \ref{section:introduction}) shows, in particular, that
given any $d > 0$ one can find a system of exponentials with
real frequencies, which is a Riesz basis in $L^2(S)$ for a
``dense'' family of multiband sets $S \subset \R$, $\mes S = d$.

Here we have discussed the problem in the periodic setting, where
an additional restriction is that the exponentials should have
integer frequencies. Theorem \ref{thm:main-yes} implies that given
any \emph{irrational} number $d \in (0,1)$ one can find a system
of exponentials with integer frequencies, which is a Riesz basis
in $L^2(S)$ for a ``dense'' family of multiband sets $S \subset \T$,
$\mes S = d$.

It would be interesting to know what can be said if $d$ is a
\emph{rational} number. Compare also with the results about
universal completeness of exponentials given in 
\cite{olevskii-ulanovskii:universal-2}.


\end{document}